\documentclass{amsart}
\usepackage[procnames]{listings}
\usepackage{cite}
\usepackage{color}
\usepackage{ amssymb }
\usepackage{amsmath}
\usepackage{mathtools}
\usepackage{amssymb}
\usepackage[toc,page]{appendix}
\numberwithin{equation}{section}
\title{Quadratic forms and semiclassical eigenfunction hypothesis for flat tori }
\author{Naser T Sardari}
\date{Spring 2016}
	
	\newtheorem{thm}{Theorem}[section]

	\newtheorem{rem}[thm]{Remark}	
	\newtheorem{lem}[thm]{Lemma}

	\newtheorem{cor}[thm]{Corollary}
	
	\theoremstyle{defi}

	\theoremstyle{pf}

	\def\vol{\text{vol}}

\begin{document}
\maketitle

\begin{abstract}

 Let $Q(X)$ be any integral primitive positive definite quadratic form in $k$ variables where $k\geq4$  and discriminant $D$. We give an upper bound on the number of integral solutions of $Q(X)=n$ for any integer $n$ in terms of $n$, $k$ and $D$. As a corollary, we give a definite answer to a conjecture of Lester and Rudnick on the small scale equidistribution of orthonormal basis of eigenfunctions restricted to  an individual eigenspace on the flat torus $\mathbb{T}^d$ for $d\geq 5$.  Another application of our main theorem gives a sharp upper bound on $A_{d}(n,t)$, the number of representation of the  positive definite quadratic form $Q(x,y)=nx^2+2txy+ny^2$ as a sum of squares of $d\geq 5$ binary  linear forms where $n- n^{\frac{1}{(d-1)}-o(1)}< t < n$. This upper bound allows us to study the local statistics of integral points on sphere.  
\end{abstract}

\tableofcontents
\section{Introduction}

\subsection{Semiclassical eigenfunction hypothesis for flat tori}
We begin by describing  the main application of this paper.  Let $\mathbb{T}^d:= \mathbb{R}^d/\mathbb{Z}^d$ be the flat torus of dimension $d\geq2$ with the laplacian operator $\Delta:=\frac{1}{(2\pi)^2 }( \frac{{\partial}^2}{\partial \theta_1}+ \dots +\frac{{\partial}^2}{\partial \theta_d})$, and $\{ \psi_i  \}$ an orthonormal basis of eigenfunctions. Marklof and Rudnick \cite{Marklof} showed that for flat torus $\mathbb{T}^d$ there is a density one subsequence of any orthonormal basis $\{\psi_i\}$ such that $|\psi_i|^2$ converges weakly to the uniform distribution in $\mathbb{T}^d$, i.e. for any continuous function $f$ on $\mathbb{T}^d$ 

$$\int_{\mathbb{T}^d} |\psi_n(x)|^2f(x) d\vol(x) \rightarrow  \int_{\mathbb{T}^d} f(x) d\vol(x)   \text  {  as $n \to \infty$.  }$$
M.V. Berry \cite{Berry, Michael} in his work on the `` Semiclassical Eigenfunction Hypothesis'' suggested to go beyond this weak convergence, and study the equidistribution of $|\psi_n|^2$  on small scale. Hezari and Riviere \cite{Hamid} established the first result on the small scale equidistribution  on rational flat tori  for balls with a radius shrinking at a polynomial rate.  Motivated by their work, Lester and Rudnick \cite{Rudnick} proved the equidistribution of a density one subsequence of $\{\psi_i\}$ in an optimal small scale on the flat torus $\mathbb{T}^d$ for $d\geq 2$. Density one means if we order the eigenfunctions by their eigenvalue then the subsequence contains a density one subset of the basis. More precisely, they show that along a density one subsequence of the orthonormal basis $\{\psi_n \}$,

\begin{equation}\label{smallscale} \lim_{n\to \infty} \sup_{B(y,r)}\Big|\frac{1}{\vol (B(y,r))}\int_{B(y,r)} |\psi_n(x)|^2 d \vol(x)-1\Big|=0,\end{equation}
where $\psi_n$ is an eigenfunction with eigenvalue $\lambda_n$ and $B(y,r)$ is any ball of radius $r > \lambda_n^{-\frac{1}{2(d-1)}+o(1)}$ centered at an arbitrary point $x \in \mathbb{T}^d $. The exponent $-\frac{1}{2(d-1)}$, for the size of the ball,  is optimal.  Moreover, for dimension 3 and 4, they prove a stronger result. They prove the small scale equidistribution holds for almost every eigenfunction in individual eigenspace with the optimal exponent  $-\frac{1}{2(d-1)} $ and they conjecture that the same result holds for every $d\geq 5$ . The main application of this paper is to resolve this conjecture. In what follows, we explain their result and conjecture in detail and how it is reduced to counting pairs of integral points with small distance on sphere.

We note that if $d\geq 4$ then the set of eigenvalues of the laplacian $\Delta$ on $\mathbb{T}^d$ are given by the set of non-negative integers $\{n\in \mathbb{Z} : n\geq 0   \}$. We denote the associated eigenspace by $E_n$. It is well-known that the dimension of this space is equal to the number of integral points on sphere of radius $\sqrt{n}$ in $\mathbb{R}^d$.  Therefore, for $d \geq 5$, the multiplicity of each eigenvalue $n$ is large and  grow like a scalar multiple of $n^{\frac{d-2}{2}}$. The intersection of the orthonormal basis of eigenfunctions $\{ \psi_i\}$ with $E_n$ gives us an orthonormal basis of eigenfunctions for $E_n$. We denote this orthonormal basis of $E_n$ by $B_n:= E_n \cap B$. Lester and Rudnick \cite{Rudnick} prove that there exists a large subset $C_n \subset B_n$, which means $|C_n|=(1-o(1))|B_n|$, such that  along any sequence of $\{\psi_n : \psi_n \in C_n \}$
\begin{equation}\label{smallscale} \lim_{n\to \infty} \sup_{B(y,r)}\Big|\frac{1}{\vol (B(y,r))}\int_{B(y,r)} |\psi_n(x)|^2 d \vol(x)-1\Big|=o(1),\end{equation}
where $B(y,r)$ is any ball of radius $r > \lambda_n^{-\frac{1}{2(d-1)}+o(1)}$ centered at an arbitrary point $x \in \mathbb{T}^d $. This means that we can choose our density one subsequence so that it contains a large proportion of the basis restricted to each eigenspace. In this paper we prove the analogue of this result for $d\geq 5$.

\begin{cor} \label{small}Let $\mathbb{T}^d:= \mathbb{R}^d/\mathbb{Z}^d$ be a $d\geq5$ dimensional flat torus and $\{ \psi_i  \}$ an orthonormal basis of eigenfunctions for the laplacian operator. Let $E_n$ be the associated eigenspace for eigenvalue  $n \in\mathbb{Z} \geq 0$ and  $B_n:= E_n \cap $ $\{ \psi_i  \}$ be the restriction of the orthonormal basis $\{ \psi_i  \}$ to $E_n$. Then there exists a large subset $C_n \subset B_n$, which means $|C_n|=(1-o(1))|B_n|$, such that  along any sequence of $\{\psi_n : \psi_n \in C_n \}$
\begin{equation}\label{smallscale} \lim_{n\to \infty} \sup_{B(y,r)}\Big|\frac{1}{\vol (B(y,r))}\int_{B(y,r)} |\psi_n(x)|^2 d \vol(x)-1\Big|=o(1),\end{equation}
where $B(y,r)$ is any ball of radius $r > \lambda_n^{-\frac{1}{2(d-1)}+o(1)}$ centered at an arbitrary point $x \in \mathbb{T}^d $. The exponent $-\frac{1}{2(d-1)}$ is optimal. In fact there exists a sequence of balls with radius $r\approx \lambda_n^{-\frac{1}{2(d-1)}-o(1)}$ and an orthonormal basis $\{ \psi_{i}\}$ where  (\ref{smallscale}) does not hold for a positive proportion of eigenfunctions in $B_n$.

\end{cor}


%

  \subsection{Local statistics of integral points on sphere}
  In this section we explain Lester and Rudnick's \cite{Rudnick}[Remark 5.4] observation that the small scale equidistribution for individual eigenspace on flat tori is related to counting pairs of integral points with small distance on sphere. 
  
  Let $\psi \in B_n$ be an orthonormal function in our basis then we can write 

\begin{equation}\label{1}
\psi(\theta)=\sum_{\lambda \in \mathbb{Z}^d, \| \lambda\|^2=n} c_{\lambda}(\psi) e(\langle \lambda,\theta \rangle),
\end{equation}
where
 \begin{equation} \sum_{\lambda \in \mathbb{Z}^d, \| \lambda\|^2=n} \|c_{\lambda}(\psi)\|^2=1.\end{equation}
 Since $B_n$ is an orthonormal basis for $E_n$, similarly if we fix $\lambda \in \mathbb{Z}^d$ we also obtain
 \begin{equation}\label{orthogonality}
  \sum_{\psi \in B_n} \|c_{\lambda}(\psi)\|^2=1.
 \end{equation}
 Lester and Rudnick estimate the indicator function of the ball $B(x,r) $ by the majorant and minorant Beurling-Selberg trigonometric polynomials $a^{\pm} (\theta)$ on the flat torus $\mathbb{T}^d$ such that

\begin{eqnarray*}
a(\theta)^{\pm}&=&\sum_{\xi \in \mathbb{Z}^d} a(\xi)e(\langle \theta, \xi \rangle),
\\
a^{\pm}(0)^{\pm}&=& \vol(B(x,r))+O(r^{d-o(1)}),
\\
\hat{a}(\xi)^{\pm}&=&0 \text{ if }  |\xi|> n^{\frac{1}{2(d-1)}-o(1)},
\\
\big| \hat{a}(\xi)^{\pm} \big| &\leq & r^d.
\end{eqnarray*}
Then one can estimate from below and above the local integral (\ref{smallscale}) by using the trigonometric polynomials $a^{\pm}$ and reduce the problem into counting pairs of integral points with small distance on the sphere. We include a brief exposition of this reduction in what follows
\begin{eqnarray*}
\frac{\int_{\mathbb{T}^d} \| \psi \|^2 a^{\pm}(\theta) d\theta }{a^{\pm}(0)}-1&=& \int \big| \sum_{\lambda}c_{\lambda}(\psi)e( \langle \lambda,\theta \rangle )    \big|^2 \big(  \sum_{\xi} \frac{\hat{a}^{\pm}(\xi)}{\hat{a}^{\pm}(0)} e( \langle \xi,\theta \rangle )  \big) d\theta
\\
&=&\sum_{\lambda-\lambda^{\prime}=\xi} c_{\lambda}(\psi)\bar{c}_{\lambda^{\prime}} \frac{\hat{a}^{\pm}(\xi)}{\hat{a}^{\pm}(0)}
\\
&\ll& \sum_{0<\|\lambda-\lambda^{\prime}\|\leq n^{\frac{1}{2(d-1)}-o(1) }} \|c_{\lambda}(\psi)\|^2+ \|c_{\lambda^{\prime}}(\psi)\|^2.
\end{eqnarray*}
 Average this inequality over $\psi \in B_n$ and use the identity (\ref{orthogonality}) to obtain 
 \begin{eqnarray*}
 \sum_{\psi \in B_n }\Big[  \sup_{B(y,r)}\Big|\frac{1}{\vol (B(y,r))}\int_{B(y,r)} |\psi(x)|^2 d \vol(x)-1\Big| \Big]
 \\
 \leq \frac{1}{|B(n)|}\sum_{0<\|\lambda-\lambda^{\prime}\|\leq n^{\frac{1}{2(d-1)}-o(1) }} 1.
 \end{eqnarray*}
 Note that the right hand side of the above inequality is the average over the pair of integral points $\lambda\neq \lambda^{\prime}$ on the sphere of radius $\sqrt{n}$ where their distance is less than $n^{\frac{1}{2(d-1)}-o(1) }$. For a large integer $n$, let

$$E(n)=E_d(n)=\{x\in \mathbb{Z}^d :|x|^2 =n\},$$
be the set of integral lattice points on the sphere $\sqrt{n}S^{d-1}$ of radius $\sqrt{n}$. For any $Y \ll \sqrt{n}$, and $|x|^2 =n$, let

$$\text{cap}(x;n,Y):=\{y \in\sqrt{n}S^{d-1} , |x-y|\leq Y\},$$
be the spherical cap of size $Y$ around the point $x$ on the sphere $\sqrt{n}S^{d-1}$.
Given a point $x \in E(n)$, let 
$$\mu(x;n,Y):=|\text{cap}(x;n,Y)\cap \mathbb{Z}^d|-1,$$
be the number of other lattice points in the cap around $x$. The mean of $\mu(x; n,Y)$, averaged over all lattice points $E(n)$ is:
$$
\left< \mu(\bullet ;n,Y)\right>:=\frac{1}{|E(n)|}\sum_{x\in E(n)}\big(|\text{cap}(x;n,Y)\cap\mathbb{Z}^d |-1\big).
$$
We note that $|E(n)| \approx n^{\frac{d-2}{2}}$ when $d\geq 5$ and the volume of $\sqrt{n}S^{d-1}$ is $\approx n^{\frac{d-1}{2}}$. Heuristically,  if we assume that the integral points are uniformly distributed,  then we expect to have $\frac{Y^{d-1}}{n^{\frac{1}{2}}}$  integral points inside a cap of size $Y$. So, if $Y \ll n^{\frac{1}{2(d-1)}-o(1)}$ then we expect to have no integral points  inside a cap of size $Y$. On the other hand, if $n^{\frac{1}{2(d-1)}+o(1)}\ll Y$ then we expect to have many points. In fact, we have the following corollary of our main theorem that makes this heuristic rigorous.
\\

\begin{cor}\label{mainn}
Let $d\geq 5$. Assume that $Y$, the size of the caps,  satisfies 

$$n^{\frac{1}{2(d-1)}+o(1)}\ll Y \ll n^{1/2}.$$
Then the probability that a cap of size $Y$ centered at integral point has more than $\log(n)$ points is greater than $1/2$
$$\mathbb{P}\big[ \mu(\bullet;n,Y)>\log(n)\big]> 1/2,$$
as a result $$\left<\mu(\bullet;n,Y)  \right> \to \infty \text{ as } n\to \infty. $$
On the other hand, if  $Y \ll n^{\frac{1}{2(d-1)}-o(1)}$, then
$$\left<\mu(\bullet;n,Y)  \right> \to 0 \text{ as } n\to \infty.$$
\\

\end{cor}
\begin{rem}
Lester and Rudnick \cite{Rudnick}[Remark 5.4]  remarked that the small scale equidistributin for individual eigenspace as stated in  Corollary \ref{small} is implied from  $$\left<\mu(\bullet;n,Y)  \right> \to 0 \text{ as } n\to \infty,$$
where $Y \ll n^{\frac{1}{2(d-1)}-o(1)}$. Bourgain \cite{Rudnick}[Theorem 4.1] show that the exponent $-\frac{1}{2(d-1)}$ is optimal by using the fact that 
$$\mathbb{P}\big[ \mu(\bullet;n,Y)>2 \big]> 1/2. $$
Therefore corollary \ref{mainn} implies the corollary \ref{small}. 
\end{rem}

\subsection{Main theorem}\label{Main theorem}
%
We begin by introducing some notations. %
Let $Q(X)$ be an integral quadratic form where $X=(x_1,\dots,x_k)$ and define 
$$A:=\Big[\frac{\partial^2 Q}{\partial x_i \partial x_j}\Big],$$
then 
$$Q(X)=1/2X^T A X.$$
Let $D:=det (A)$ be the discriminant of $Q$. We write $r(Q,n)$ for the number of integral solutions of 
$$Q(X)=n.$$ 
We consider the Theta series associated to this quadratic form
$$\Theta(z)=\sum_{n}r(Q,n)e(nz).$$
This is a modular form of weight $k/2$ and level $N$, where $N$ is the smallest integer such that $NA^{-1}$ is an even integral matrix . By the theory of modular forms we can write $\Theta(z)$ uniquely as a sum of standard Eisenstein series $E(z,Q)$ (the Eisenstein series associated to $Q$) and a cusp form $F(z,Q)$

 $$\Theta(z)=E(z,Q)+F(z,Q).$$
 From this decomposition
$$r(Q,n)=\rho(n,Q)+\tau(n,Q).$$
  where, $\rho(n,Q)$ and $\tau(n,Q)$ are the $n$-th Fourier coefficients of $E(z,Q)$ and $F(z,Q)$ respectively. We use the spectral theory of automorphic forms and bounds on the Fourier coefficients of modular forms to prove Theorem (\ref{main}).  In our main theorem, we give a uniform upper bound on the number of integral points on a quadric that is defined by any positive definite integral primitive quadratic form  in  $k\geq 4$ variables and with discriminant $D$ . Similar uniform results in a different context (explicit bounds for representability by a quadratic form) has been extensively studied by various authors. In particular, there is PhD work of Hanke \cite{Hanke} who uses theta series to get estimates which are uniform in the coefficients and also the work of Schulze-Pillot \cite{Schulze}. More recently,  Browning and Deitmann \cite[Proposition 1]{Browning} established a result that recovers our theorem for  in the generic situation where the coefficients of the quadratic form is of order $D^{1/k}$. This result is not enough to establish the small scale equidistribution on rational flat tori.  We need a uniform result for all quadratic forms with discriminant $D$ that is stated in Theorem 1.4. We explain this in Remark \ref{remarks} after stating the our main theorem.
\begin{thm}\label{main}
Let $n$ be any integer, and let $Q(X)$ be any  primitive positive definite integral quadratic form of discriminant $D$ in  $k\geq 4$ variables. If  $D \ll n^{\frac{k-3}{2(k-2)}}$, then the number of integral solutions of $Q(X)=n$ is bounded from above by 
\begin{equation}
c_{\epsilon}\frac{n^{\frac{k-2}{2}}}{\sqrt{D}}\gcd(D,n)^{1/2}n^{\epsilon},
\end{equation}
where $c_{\epsilon}$ is a constant which depends only on $\epsilon$ and not on $Q(X)$ or $n$.
\\

\end{thm}  
\begin{rem}\label{remarks}
Corollary (\ref{small}) is a consequence of Theorem (\ref{main}) with the discriminant bound $D\ll n^{1/4}$ and no conditions on the height of the quadratic forms.
 In fact the quadratic forms that we deal with are coming form the lattices given by the hyperplanes orthogonal to integral vectors of square norm $D$. So, the height of the quadratic forms might be as big as $D$. For quadratic forms in 5 or more variables in Theorem (\ref{main}), we do not need to appeal to Blomer's result stated in the appendix. 
\end{rem}

%

%
%
%
%
%
%
%
%
%
%

%

\subsection{Outline of the paper}

We give a brief outline of this paper. In section (\ref{sec}), we show that  corollary (\ref{mainn}) is a consequence of our main Theorem (\ref{main}).   Next, we give a proof of our main Theorem (\ref{main}). In the proof of Theorem (\ref{main}), we use an improved version of a lemma in Blomer's papers \cite[Lemma 4.2] {Val} or \cite[Lemma 3]{VB}. Professor Blomer provided us  a proof for this improved version of his previous lemma.We include his proof in our appendix.  We are responsible for any gap or typo in the appendix. In Lemma~(\ref{1}), we use the Siegel product formula (The main term of the Hardy-Littlewood formula) to give an upper bound on $\rho(n,Q)$. In Lemma~(\ref{Cusp}), we invoke the upper bound of Blomer \cite[Lemma~4.2]{Val}  on $\|F(z,Q)\|^2$ and then we apply the Petersson trace formula to  give an upper bound on $\tau(n,Q)$. The theorem is a consequence of lemma~(\ref{1}) and lemma~(\ref{Cusp}). 
  
%
%

\subsection{Acknowledgments}  We are grateful to Professor Valentin Blomer for his comments and letter to us.  In the letter,  he proves lemma~(\ref{l2}) which improves his earlier lemma \cite[Lemma~4.2]{Val}. This lemma is crucial  in our work.   We would also like to thank Professor Zeev Rudnick for suggesting  this project to us and his comments on the earlier version of this paper. Finally, I would like to thank Masoud Zargar for several comments and remarks on the earlier versions of this work.

%
%
%
%
%
%

%
%
%
%
%
\section{Proof of corollary (\ref{mainn}) }\label{sec}

\begin{proof}
\noindent We begin by proving  the first part of Corollary (\ref{mainn}) when the size of the caps is large, i.e.

$$n^{\frac{1}{2(d-1)}+o(1)}\ll Y \ll n^{1/2}.$$
This part is elementary; we use a covering argument in combination with a pigeonhole argument.
We assume that $n^{\frac{1}{2(d-1)}+o(1)}\ll Y$. Call an integral point $p\in E(n) $ \emph{bad} if 
$$\mu(p;n,Y)\leq \log(n).$$
We denote the number of bad points by $B$. Assume to the contrary that $B\geq \frac{1}{2}|E(n)|$. Hence, using $|E(n)|\approx n^{\frac{d-2}{2}}$, $B \gg n^{\frac{d-2}{2}}.$
Consider  balls of radius $\frac{Y}{2}$ centered at each bad point. Each point of the sphere is covered by at most $\log(n)$ of these balls; otherwise, there are more than $\log(n)$ bad points with distance at most $Y$, contradicting the definition of a bad point. Therefore, we have the following inequality from a covering argument

$$B(Y/2)^{d-1} \ll n^{\frac{d-1}{2}} \log(n). $$
Hence,
$$Y^{d-1}\ll n^{1/2}\log(n).$$
This is a contradiction to $Y\gg n^{\frac{1}{2(d-1)}+o(1)}.$  Therefore, $B<\frac{1}{2}|E(n)|$. Consequently,
$$\mathbb{P}[\left<\mu(\bullet;n,Y)  \right> > \log(n)]>1/2, $$
Therefore,
 $$ \left<\mu(\bullet;n,Y)  \right> \to \infty \text{ as }n\to\infty.$$ 
\noindent This concludes the proof of the first part.
\\

 \begin{lem}\label{ll}
 
The second part of the Corollary (\ref{mainn}) is a consequence of the inequality
$$A_{d}(n,t) \ll n^{\frac{d-3}{2}+\varepsilon} \gcd(n,t)^{1/2} (n-t)^{\frac{d-3}{2}},$$
where $n- n^{\frac{1}{(d-1)}-o(1)}<t<n$ and $5\leq d$. 
\end{lem}

\begin{proof}
For the second part of Corollary (\ref{mainn}), when the size of caps is small, i.e.  
$$Y \ll n^{\frac{1}{2(d-1)}-o(1)} ,$$
\noindent we follow the technique in the argument for \cite[lemma 11 ]{Rudnick} to prove lemma~(\ref{ll}). Let $A_{d}(n,t)$ be the number of ordered pairs of distinct integral lattice points $(p,q)\in\mathbb{Z}^d\times\mathbb{Z}^d$ such that $|p|^2=|q|^2=n$ and $|p-q|^2=2(n-t)$. Note that a change of summation argument gives us the equality 
$$\left<\mu(\bullet;n,Y)\right>=\frac{1}{|E(n)|}\sum_{n-\frac{Y^2}{2}< t<n}A_d(n,t).$$ 
Since $Y \ll n^{\frac{1}{2(d-1)}-o(1)} $ and $n-\frac{Y^2}{2}< t<n$, so  $n- n^{\frac{1}{(d-1)}-o(1)}<t<n$ and we can apply the inequality
$$A_{d}(n,t) \ll n^{\frac{d-3}{2}+\varepsilon} \gcd(n,t)^{1/2} (n-t)^{\frac{d-3}{2}},$$
and we obtain 
\begin{eqnarray*}
\left<\mu(\bullet;n,Y)\right>&\ll& \frac{1}{|E(n)|}\sum_{n-\frac{Y^2}{2}< t<n}n^{\frac{d-3}{2}+\varepsilon} \gcd(n,t)^{1/2} (n-t)^{\frac{d-3}{2}}\\
&\ll& \frac{n^{\frac{d-3}{2}+\varepsilon}}{|E(n)|}\sum_{n-\frac{Y^2}{2}< t<n}\gcd(n,t)^{1/2} (n-t)^{\frac{d-3}{2}}\\
&\ll& \frac{n^{\frac{d-3}{2}+\varepsilon}}{|E(n)|}\sum_{1< s<\frac{Y^2}{2}}\gcd(n,s)^{1/2} s^{\frac{d-3}{2}}\ (\text{where } s:=n-t)\\
&\ll& \frac{n^{\frac{d-3}{2}+\varepsilon}Y^{d-1}n^{\varepsilon}}{|E(n)|}.
\end{eqnarray*}
\noindent Using $|E(n)|\approx n^{\frac{d-2}{2}}$ and $Y\ll n^{\frac{1}{2(d-1)}-o(1)}$, we get
$$\left<\mu(\bullet;n,Y)\right>\ll n^{\varepsilon-o(1)}\to 0\text{ as }n\to\infty.$$
\end{proof}
\noindent Lemma (\ref{lem}) is devoted to showing that the following inequality in lemma~(\ref{ll}) is true:
$$A_{d}(n,t) \ll n^{\frac{d-3}{2}+\varepsilon} \gcd(n,t)^{1/2} (n-t)^{\frac{d-3}{2}},$$
for $n- n^{\frac{1}{(d-1)}-o(1)}<t<n$ and $ 5 \leq d$.\\
\\
\begin{lem}\label{lem}
Theorem (\ref{main}) implies that 
$$A_{d}(n,t) \ll n^{\frac{d-3}{2}+\varepsilon} \gcd(n,t)^{1/2} (n-t)^{\frac{d-3}{2}},$$
where $n- n^{\frac{1}{(d-1)}-o(1)}<t<n$ and $ 5 \leq d$.\\
\end{lem}
\begin{proof}
Recall that $A_d(n,t)$ is the number of ordered pairs of distinct integral lattice points $(p,q)\in\mathbb{Z}^d\times\mathbb{Z}^d$ such that $|p|^2=|q|^2=n$ and $|p-q|^2=2(n-t)$. Let $v:=p-q$. Note that $v$ is an integral vector of length $\sqrt{2(n-t)}$ and so, up to a constant, we have at most
\begin{equation}\label{choice}(n-t)^{\frac{d-2}{2}},\end{equation}
choices for $v$, because $d\geq 5$ . Note that
$$\left<p+q,v \right>=0,$$ 
and
$$|p+q|^2=2(n+t).$$
Let $L$ be the $(d-1)$-dimensional lattice that is given by the intersection of $\mathbb{Z}^d$ and the hyperplane orthogonal to $v$. Therefore, $p+q$ lies inside the lattice $L$ and the sphere of radius $\sqrt{2(n+t)}$.  The fundamental domain of the lattice $L$ has volume $|v^{\prime}|$ where $v^{\prime}$ is the primitive integral vector in the direction of $v$. We take an integral basis for the lattice $L$ and denote it by $$\{e_1,\dots, e_{d-1}   \}.$$
We define the symmetric matrix $A$ by 
$$A:=\big[ \left< e_i,e_j \right>_{1 \leq i,j \leq d-1}  \big].$$
We define the quadratic form $Q(X)$ by
$$Q(X):=X^{T}AX.$$
Let $k:=(d-1)$. This quadratic from is a primitive quadratic form in $k \geq 4$ variables with discriminant $D:=|v^{\prime}|^2$. We have the following upper bound on the discriminant of $Q$  

$$D=|v^{\prime}|^2 \leq n^{\frac{1}{k}-o(1)}.$$
Since $ \frac{1}{k} \leq \frac{k-3}{2(k-2)}$ when $k \geq 4$ then
$$D \ll n^{\frac{k-3}{2(k-2)}}.$$ 
We can apply  theorem \ref{main}, and as a consequence we have the following bound on the number of lattice points of $L$ with norm $2(n+t)$
$$c_{\epsilon}\frac{(2(n+t))^{\frac{d-3}{2}}}{|v^{\prime}|}\gcd(2(n-t),2(n+t))^{1/2}(2(n+t))^{\epsilon}\ll\frac{n^{\frac{d-3}{2}}}{|v^{\prime}|}\gcd(n,t)^{1/2}n^{\epsilon}.$$
Recall that $|v|^2=|p-q|^2=2(n-t)$ and the number of integral points $v$ where $\frac{|v|}{|v^{\prime}|}=l$ is less than $\big(\frac{{n-t}}{l^2}\big)^{\frac{d-2}{2}+\epsilon}$ up to a constant. We give an upper bound on $A_d(n,t)$ by first choosing $v=p-q$ and then $p+q$
$$A_{d}(n,t) \ll \sum_{l^2|2(n-t)}(\frac{n-t}{l^2})^{\frac{d-2}{2}}\frac{n^{\frac{d-3}{2}}}{\frac{\sqrt{n-t}}{l}}\gcd(n,t)^{1/2}n^{\epsilon}.$$
Therefore,
$$A_{d}(n,t) \ll n^{\frac{d-3}{2}+\epsilon} \gcd(n,t)^{1/2} (n-t)^{\frac{d-3}{2}}.$$

\end{proof}

\end{proof}

\section{Proof of the Main theorem}
We give a brief plan of the proof of Theorem (\ref{main}) in what follows. Recall the notations that were introduced in section (\ref{Main theorem}). 
%
%
%
  In Lemma~\ref{1}, we use the Siegel product formula (the main term of the Hardy-Littlewood formula) to give an upper bound on $\rho(n,Q)$. In Lemma~\ref{Cusp}, we invoke the upper bound of Blomer on $\|F(z,Q)\|^2$ in Lemma (\ref{Valup}) that is proved in our appendix. Finally we apply the Petersson trace formula together with a Cauchy inequality to  give an upper bound on the $n$-th Fourier coefficient  $\tau(n,Q)$ of $F(z,Q)$. Theorem (\ref{main}) is a consequence of lemma~(\ref{1}) and lemma~(\ref{Cusp}). 
\\
%

The following lemma proved by Blomer  \cite[page 6]{VB} for $k=3$.  We follow his strategy and give a proof for every $k \geq 4$.

\begin{lem}\label{1}\label{Val} 
We have the following upper bound on the $n$-th Fourier coefficient of  the Eisenstein series of the form $Q$  
$$\rho(n,Q)\leq c_{\epsilon}\frac{n^{\frac{k-2}{2}}}{\sqrt{D}}\gcd(N,n)^{1/2}(nN)^{\epsilon},$$
where $N$ is the level of the quadratic from $Q$.
\\
\end{lem}
%
%
%

\begin{proof}  $\rho(n,Q)$, the $n$-th Fourier coefficient of the Eisenstein series,   coincides with the main term of the Hardy-Littlewood formula. The main term of the Hardy-Littlewood formula, is given by the product of local densities 
$$\rho(n,Q):=n^{\frac{k-2}{2}}\sigma_{\infty}\prod_{p}\sigma_{p} .
$$
We have the following formula for the local densities: $$\sigma_p=\sum_{t=0}^{\infty} S(p^t),$$
where
$$S(p^t):=\frac{1}{p^{tk}} \sum_{a}^* \sum_{b}e\Big(\frac{a(Q(b)-n)}{p^t}\Big).$$
$\sum_{a}^*$ means that  $a$ varies mod $p^{t}$ where $\gcd(a,p)=1$,  $b$ is a vector that varies mod $p^t\mathbb{Z^d}$ and S(1)=1.

We give an upper bound for each place. First, we start with $\infty$.
The density at $\infty$ is given by 
$$\sigma_{\infty}=\lim_{\epsilon \to 0} \frac{\int_{1<Q(X)<1+\epsilon}1   dx_1     \dots  dx_k}{\epsilon}.
$$
We diagonalize $Q(X)$ in an orthonormal coordinates $Y:=(y_1, \dots, y_k)$ such that
$$Q(Y):=n_1 y_1^2+ \dots +n_k y_k^2.$$ 
where $n_1, \dots, n_k$ are the eigenvalues of the symmetric matrix $A$.
We change the variables to $(y_1, \dots, y_k)$ to get
$$\sigma_{\infty}=\frac{1}{\sqrt{det{A}}}\vol({S^{k-1}}).$$
Next, we give an upper bound on  the local densities $\sigma_p$ where $p\neq 2$. Since $p$ is an odd prime number, we can diagonalize our quadratic form $Q(X)$ over the local ring $\mathbb{Z}_{p}$. Without loss of generality we assume that
$$Q(x_1,\dots,x_k)=\sum_{i=1}^{k}a_ip^{\alpha_i}x_i^2,$$
where $\gcd(a_i,p)=1$ and $a_i\in \mathbb{Z}_p$.
We substitute the diagonal expansion of $Q(x_1,\dots,x_k)$ to compute $S(p^t)$
\begin{equation*}
\begin{split}
S(p^t)&:=\frac{1}{p^{tk}}\sum_{a}^*\sum_{b\in{(\frac{\mathbb{Z}}{p^t\mathbb{Z}}})^k}e\Big(\frac{a(Q(b)-n)}{p^t}\Big)
\\
&=\frac{1}{p^{tk}}\sum_{a}^*\sum_{b\in{(\frac{\mathbb{Z}}{p^t\mathbb{Z}}})^k}e\Big(\frac{a(\sum_{i=1}^{k} a_{i}p^{\alpha_i}b_i^2-n)}{p^t}\Big)
\\
&=\frac{1}{p^{tk}}\sum_{a}^*e\Big(\frac{-an}{p^t}\Big)\prod_{i=1}^{k}\sum_{b \text{ mod } p^t}e\Big(\frac{aa_{i}p^{\alpha_i}b^2}{p^t}\Big).
\end{split}
\end{equation*}
We note that the last summation is a Gauss sum. Let $G(h,m):=\sum_{x \text{ mod } m} e(\frac{hx^2}{m})$ be the Gauss sum, and let $\epsilon_{m}=1$ if $m\equiv1 (\text{ mod } 4) $ and $\epsilon_{m}=i$ if $m\equiv 3 (\text{ mod } 4) $. Then if $\gcd(h,m)=1$ we have

\begin{equation*}
G(h,m):=\begin{cases} 
\epsilon_m \Big(\frac{h}{m} \Big) m^{1/2} & \text{ if } m \text{ is odd },\\
(1+\chi_{-4}(h))m^{1/2}  &\text{ if } m=4^{\alpha},\\
(\chi_8(h)+i\chi_{-8}(h)) m^{1/2}  &\text{ if } m=2.4^{\alpha}, \alpha\geq 1.
\end{cases}
\end{equation*}
%
where $\Big(\frac{h}{m} \Big)$ is the Jacobi symbol. We have
$$S(p^t)=\frac{1}{p^{tk}}\sum_{a}^*e\Big(\frac{-an}{p^t}\Big)\prod_{i=1}^{k}p^{\min(\alpha_i,t)}G(aa_i,p^{t-\alpha_i}).$$
We define $G(\bullet,p^{t-\alpha_i}):=1$ when $t <\alpha_i $. 
We substitute the values of $G$ and obtain,

\begin{equation*}
\begin{split}
|S(p^t)|=\frac{\prod_{i=1}^k p^{\min(\frac{\alpha_i+t}{2},t)}}{p^{tk}}\big|\sum_{a}^*e\Big(\frac{-an}{p^t}\Big)\Big(\frac{a}{p} \Big)^{k^{\prime}}\big|,
\end{split}
\end{equation*}
where $k^{\prime}$ is the number of integers $i$ such that $1 \leq i \leq k$ and $t-\alpha_i$ is a positive odd integer. Assume that $n=p^{\beta}n^{\prime}$, where $\gcd(n^{\prime},p)=1$. If $k^{\prime}$ is an odd number then the inner sum is a Gaussian sum, and we obtain
\begin{equation*}
 \big|\sum_{a \text{ mod } p^t}^*e\Big(\frac{-ap^{\beta}n^{\prime}}{p^t}\Big)\Big(\frac{a}{p} \Big)\big|=
 \begin{cases}
 p^{t-\frac{1}{2}}   &\text{  if  } \beta=t-1,\\
 0    &\text{  otherwise  }.
 \end{cases} 
\end{equation*}
Hence, if $k^{\prime}$ is odd we deduce that
\begin{equation}\label{odd}
|S(p^t)|= \begin{cases}\frac{\prod_{i=1}^k p^{\min(\frac{\alpha_i+t}{2},t)}}{p^{tk}}p^{t-\frac{1}{2}} &\text{ if } \beta=t-1  ,\\
0 &\text{ otherwise. }
\end{cases}
%
\end{equation}
On the other hand, if $k^{\prime}$ is even then the inner sum is a Ramanujan sum $c_{p^t}(n)$,
$$c_{p^t}(n)=\sum_{a}^*e\Big(\frac{-an}{p^t}\Big)=\begin{cases} 0   &\text{ if } \beta< t-1, \\ 
-p^{t-1} &\text{  if  } \beta=t-1, \\
\phi(p^t) &\text{  if  } \beta\geq t.
    \end{cases}$$
Hence, if $k^{\prime}$ is even we deduce that
 
\begin{equation}\label{even}
|S(p^t)|=\begin{cases} 0   &\text{ if } \beta< t-1, \\ 
-\frac{\prod_{i=1}^k p^{\min(\frac{\alpha_i+t}{2},t)}}{p^{tk}}p^{t-1} &\text{  if  } \beta=t-1, \\
\phi(p^t)\frac{\prod_{i=1}^k p^{\min(\frac{\alpha_i+t}{2},t)}}{p^{tk}} &\text{  if  } \beta\geq t.
\end{cases}
\end{equation}

%
Without loss of generality suppose that $\alpha_1\leq \alpha_2 \leq \dots \leq \alpha_k$. Since $Q(X)$ is primitive we deduce that $\alpha_1=0$.
If $k^{\prime}$ is odd (\ref{odd}) and if $k^{\prime}$ is  even (\ref{even}), we deduce that
\begin{equation}\label{inq}
\begin{split}
|\sigma_p| &\leq \sum_{t=0}^{\beta+1} p^{t/2}\frac{\prod_{i=2}^k p^{\min(\frac{\alpha_i+t}{2},t)}}{p^{tk}}p^{\min(t,\frac{t+\beta}{2})}\\
&\leq \sum_{t=0}^{\beta+1}\frac{p^{t/2}p^{(k-1)\min(\frac{\alpha_k+t}{2},t)}p^{\min(t,\frac{t+\beta}{2})}}{p^{tk}}.
\end{split}
\end{equation}
From the above inequality and the condition that $k\geq 4$, we give an upper bound on the product of local densities $\prod_{p\nmid 2nN}\sigma_p$ as follows.
\begin{equation}
\begin{split}
\prod_{p\nmid 2nN}\sigma_p&\leq \prod_{p\nmid 2nN}(1+p^{-\frac{k-1}{2}})
\\
&\leq \sum_{n=1}^{\infty}\frac{1}{n^{3/2}}=O(1).
\end{split}
\end{equation}
If $p$ is an odd prime number from the inequality (\ref{inq}), we obtain
$$\sigma_p\ll p^{\frac{\min(\alpha_k,\beta)}{2}}.$$
Hence, for odd prime numbers where $p|nN$ we have
\begin{equation}\label{rrr}
\prod_{p|nN}\sigma_p\leq (nN)^{\epsilon}\gcd(n,N)^{1/2}.
\end{equation}
Finally, we assume that $p=2$. One can write any quadratic form $Q(X)$, after a change of variables over the 2-adic integers $\mathbb{Z}_2$, as a direct sum of scalar multiples of $q_0(x):=x^2$, $q_1(x_1,x_2):=x_1x_2$ and $q_2(x_1,x_2):=x_1^2+x_1x_2+x_2^2$.  The corresponding Gauss sums can be evaluated for odd $h$:

$$\sum_{b_1, b_2 \text{ mod } 2^t} e\big(\frac{2^{\alpha_1}x_1x_2h}{2^t}    \big)=2^{\min(t+\alpha_1,2t)},$$
and
$$\sum_{b_1, b_2 \text{ mod } 2^t} e\big(\frac{2^{\alpha_1}(x^2+x_1x_2+x^2)h}{2^t}    \big)=2^{\min(t+\alpha_1,2t)}.$$
We substitute the values of these Gauss sums and  obtain
\begin{equation}\label{inq22}
\begin{split}
\sigma_2(Q,n)\leq \sum_{t=0}^{\beta+1} 2^{t/2}\frac{\prod_{i=2}^k 2^{\min(\frac{\alpha_i+t}{2},t)}}{2^{tk}}2^{\min(t,\frac{t+\beta}{2})}\\
\leq \sum_{t=0}^{\beta+1} \frac{2^{t/2}2^{(k-1)\min(\frac{\alpha_k+t}{2},t)}2^{\min(t,\frac{t+\beta}{2})}}{2^{tk}}.
\end{split}
\end{equation}
From the inequality (\ref{inq22}), we obtain

$$\sigma_2(Q,n)\ll 2^{\frac{\min(\alpha_k,\beta)}{2}}.$$
This inequality together with inequality (\ref{rrr}) for odd primes $p$ implies

$$
\prod_{p|\gcd(D,N)}\sigma_p \leq  \gcd(N,D)^{1/2}.
$$
Therefore, we conclude the lemma.
\\
\end{proof}

 In Lemma~\ref{Cusp}, we invoke a result of Blomer, proved in Appendix 1, and then we apply the Petersson trace formula together with a Cauchy inequality to  give an upper bound on $\tau(n,Q)$. Theorem (\ref{main}) is a consequence of lemma~(\ref{1}) and lemma~(\ref{Cusp}). 

\begin{lem}\label{Cusp} Let $Q(X)$ be a quadratic form in $k\geq 4$ variables and discriminant $D<n^{\frac{k-3}{2(k-2)}}$. Then we have the following upper bound on the $n$-th Fourier coefficient of the cusp form part of the theta series associated to $Q$
$$|\tau(n,Q)|\ll_{d,\epsilon} D^{\frac{(k-3)}{2}}n^{(k-1)/4}\gcd(n,D)^{1/4}n^{\epsilon}.$$
In particular
$$|\tau(n,Q)|\ll c_{\epsilon}\frac{n^{\frac{k-2}{2}}}{\sqrt{D}}\gcd(D,n)^{1/4}n^{\epsilon}.$$
\end{lem}
\begin{proof}
We take an orthonormal basis  $\{f_i  \}$ for the space of cusp forms  of  weight $k/2$ and level $N$. We write $F(z,Q)$ as a linear combination of them
$$F(Q,z)=\sum_{f_i}C(Q,f_i)f_i(z).$$
We obtain
$$\tau(n,Q)=\sum_{f_i}C(Q,f)\rho_{f_i}(n).$$
We apply Cauchy inequality to obtain
\begin{equation}\label{kkk}
|\tau(n,Q)|^2\leq\big(\sum_{f_i}|C(Q,f)|^2\big) \big(\sum_{f_i} |\rho_{f_i}(n)|^2  \big).\end{equation}
Since we take an orthonormal basis for the space of cusp forms we deduce that 
$$||F||^2=\sum_{f_i}|C(Q,f_i)|^2.$$
\\
We invoke a result of Blomer. The proof is included in Appendix 1.
\begin{lem}\label{l2}\label{Valup}
We have
$$\|F\|^2 \ll N^{\epsilon}(N^{k-2} +N^{\frac{k-3}{2}}D^{1/2}+D^{1-1/k} ). $$
for $k\geq 4$. As a corollary, 
\begin{equation}\label{kk}
\|F\|^2 \ll N D^{k-3+\epsilon}.
\end{equation}
\end{lem}

%

Next, we give an upper bound on the sum over the square norm of the $n$-th Fourier coefficients of  $\{f_i\}$

 $$\big(\sum_{f_i} |\rho_{f_i}(n)|^2  \big).$$
 We apply the Petersson trace formula for the modular forms of weight $k/2$ . See \cite[{Lemma~1} ]{Hen} 
$$\frac{\Gamma(k/2-1)}{(4\pi n)^{k/2-1}}\sum_{f_i} |\rho_{f_i}(n)|^2= 1+2\pi i^{-k/2}\sum_{c\equiv0 \text{ mod } N} c^{-1}J_{k/2-1}\big(\frac{4\pi n}{c} \big)K(n,n;c).$$
 In order to give an upper bound on $K(n,n;c)$, we apply the bound on Salie's sum for odd dimensions $k$ and Weil's bound on the  Kloosterman's sum for even dimensions $k$. We invoke the following formula \cite[Corollary 14.24 ]{Iwan}  
$$ \frac{\Gamma(l-1)}{(4\pi n)^{l-1}}\sum_{f_i} |\rho_{f_i}(n)|^2 = 1+ O\Big(\tau_3(n)\gcd(n,N)^{1/2}n^{1/2}\frac{\tau(N)}{N\sqrt{l}} \log\big(1+\frac{n^{1/2}}{\sqrt{Nl}}   \big)   ),
$$where $l=k/2$. Since the discriminant $D<n^{\frac{k-3}{2(k-2)}}$ and $N|D$ then 
\begin{equation}\label{k}
\sum_{f_i} |\rho_{f_i}(n)|^2 \ll_{\epsilon}\frac{n^{(k-1)/2}}{N}n^{\epsilon}\gcd(n,D)^{1/2}.
\end{equation}
From the inequalities (\ref{kkk}), (\ref{kk}) and (\ref{k}) we deduce that 
$$|\tau(n,Q)|^2\ll_{d,\epsilon} D^{k-3}n^{(k-1)/2}\gcd(n,D)^{1/2}n^{\epsilon}.$$
This concludes the first part of the lemma.
For the second part,  $D \ll n^{\frac{k-3}{2(k-2)}}$ implies that
$$D^{\frac{k-3}{2}}n^{\frac{k-1}{4}} \leq \frac{n^{\frac{k-2}{2}}}{\sqrt{D}}.$$
Hence,
$$|\tau(n,Q)|\ll c_{\epsilon}\frac{n^{\frac{k-2}{2}}}{\sqrt{D}}\gcd(D,n)^{1/4}n^{\epsilon}.$$
The main theorem (\ref{main}) is a consequence of lemma~(\ref{l2}) and lemma~(\ref{1}).
\end{proof}

%
%
%

%

\section{Appendix}
In this appendix, we include Valentin Blomer's paper to us. We are responsible for any gap or typo in this section.
\\
Let $$Q(X)=\frac{1}{2} X^{T}AX.$$ be a primitive positive definite integral $k$-dimensional quadratic form $(k\geq 3)$ of discriminant $D$ and level $N$. For $n \in \mathbb{N}$ let $r(n,Q)$ denote the number of representations of $n$ by $Q$ and let $\rho(n,Q)$ be the main term given by a formal application of the circle method.
%
%
Denote by $\mu_1, \dots \mu_k$ the successive minima (see \cite[Chapter 12]{Cassels} ) of $Q(X)$. We can write

$$Q(X):=h_1(x_1+c_{12}x_2+\dots+c_{1n}x_k)^2+\dots+h_kx_k^2,$$
where $c_j,h_j\in \mathbb{Q}$ such that
$$h_j\asymp \mu_j. $$
Let $\Theta(Q,z)$ denote the corresponding theta-series and $F(z,Q)$ the orthogonal projection onto the space of cusp forms. The following lemma uses only reduction theory and the Lipschitz principle.
\begin{lem}
We have
$$r(Q,n)\ll \Big(1+\frac{n^{1/2}}{\mu_3}+\frac{n}{(\mu_3\mu_4)^{1/2}}+\dots+\frac{n^{\frac{k-2}{2}}}{(\mu_3\dots \mu_k)^{1/2}}          \Big)n^{\epsilon},$$
and
$$\sum_{n\leq x}r(Q,n)\ll \max_{j\leq k} \frac{x^{\frac{j}{2}}}{(\mu_1\dots \mu_j)^{1/2}}.$$

\end{lem}
\begin{proof}
We choose $x_k,x_{k-1},\dots, x_3$ in 
$$\ll \big(1+\frac{n}{\mu_k}\big)^{1/2}\dots \big( 1+\frac{n}{\mu_3}  \big)^{1/2},$$
ways. Then we are left with an inhomogeneous binary problem that has $O(n^{\epsilon})$ solutions (see \cite[Lemma 3a]{Michel}), uniformly for all choices of $x_k,\dots,x_3$. To prove the second part, we choose $x_k,\dots,x_1$ in 
$$\ll (1+\frac{x}{\mu_k})^{1/2}\dots (1+\frac{x}{\mu_1})^{1/2},$$
ways. 
\end{proof}
\begin{cor}\label{cor1}
We have
$$\sum_{n\leq x} r(Q,n)^2 \ll x^{\epsilon} \Big(x^{k-2}+\frac{x^{k-\frac{3}{2}}}{D^{1/2}}+ \frac{x^{k-1}}{D^{1-1/k}}   \Big).$$
\\
\end{cor}
Finally, we give a proof of Lemma~(\ref{Valup}). It suffices to show that
$$\|F(z,Q)\|^2 \ll N^{\epsilon}  (N^{k-2}+N^{\frac{k-3}{2}}D^{1/2}+D^{1-1/k}).  $$
Note that if  $k\geq 4$, then the right hand side is $\ll ND^{k-3+\epsilon}$.

\begin{proof}
We follow the argument of \cite[Lemma 4.2]{Val} or \cite[Lemma 3]{VB}. This gives
$$\|F(z,Q) \|^2 \ll N^{\epsilon}\Big(\frac{1}{N}+\int_{1/2}^{\infty} \sum_{n=1}^{\infty}\big( r(\hat{Q},n)^2+r_{Eis}(\hat{Q},n)^2   \big)e^{\frac{-4\pi ny}{N}} y^{k/2-2} dy   \Big),$$
where $\hat{Q}(X)=\frac{1}{2}X^T(NA^{-1})X$ has determinant $N^kD^{-1}$ and level $N$. Inserting Lemma \ref{1} and Corollary \ref{cor1}, we obtain after a short computation the desired bound.

\end{proof}

\bibliographystyle{alpha}
\bibliography{Preprint}

\begin{thebibliography}{Han04}

\bibitem[BD08]{Browning}
T.~D. Browning and R.~Dietmann.
\newblock On the representation of integers by quadratic forms.
\newblock {\em Proceedings of the London Mathematical Society}, 96(2):389--416,
  2008.

\bibitem[Ber77]{Berry}
M.V. Berry.
\newblock { Regular and irregular semiclassical wave functions }.
\newblock {\em J.Phys.A}, (10), 1977.

\bibitem[Ber83]{Michael}
Michael Berry.
\newblock Semiclassical mechanics of regular and irregular motion.
\newblock In {\em Chaotic behavior of deterministic systems ({L}es {H}ouches,
  1981)}, pages 171--271. North-Holland, Amsterdam, 1983.

\bibitem[Blo04]{Val}
Valentin Blomer.
\newblock Uniform bounds for {F}ourier coefficients of theta-series with
  arithmetic applications.
\newblock {\em Acta Arith.}, 114(1):1--21, 2004.

\bibitem[Blo08]{VB}
Valentin Blomer.
\newblock Ternary quadratic forms, and sums of three squares with restricted
  variables.
\newblock In {\em Anatomy of integers}, volume~46 of {\em CRM Proc. Lecture
  Notes}, pages 1--17. Amer. Math. Soc., Providence, RI, 2008.

\bibitem[BM13]{Michel}
Valentin Blomer and Philippe Michel.
\newblock Hybrid bounds for automorphic forms on ellipsoids over number fields.
\newblock {\em J. Inst. Math. Jussieu}, 12(4):727--758, 2013.

\bibitem[Cas78]{Cassels}
J.~W.~S. Cassels.
\newblock {\em Rational quadratic forms}, volume~13 of {\em London Mathematical
  Society Monographs}.
\newblock Academic Press, Inc. [Harcourt Brace Jovanovich, Publishers],
  London-New York, 1978.

\bibitem[Han04]{Hanke}
Jonathan Hanke.
\newblock Local densities and explicit bounds for representability by a
  quadratic form.
\newblock {\em Duke Math. J.}, 124(2):351--388, 08 2004.

\bibitem[HR15]{Hamid}
H.~{Hezari} and G.~{Riviere}.
\newblock {Quantitative equidistribution properties of toral eigenfunctions}.
\newblock {\em Accepted for publication by J. Spectral Theory}, March 2015.

\bibitem[IK04]{Iwan}
Henryk Iwaniec and Emmanuel Kowalski.
\newblock {\em Analytic number theory}, volume~53 of {\em American Mathematical
  Society Colloquium Publications}.
\newblock American Mathematical Society, Providence, RI, 2004.

\bibitem[Iwa87]{Hen}
Henryk Iwaniec.
\newblock Fourier coefficients of modular forms of half-integral weight.
\newblock {\em Invent. Math.}, 87(2):385--401, 1987.

\bibitem[LR16]{Rudnick}
Stephen Lester and Ze{\'e}v Rudnick.
\newblock Small scale equidistribution of eigenfunctions on the torus.
\newblock {\em Communications in Mathematical Physics}, pages 1--22, 2016.

\bibitem[MR12]{Marklof}
Jens Marklof and Ze{\'e}v Rudnick.
\newblock Almost all eigenfunctions of a rational polygon are uniformly
  distributed.
\newblock {\em J. Spectr. Theory}, 2(1):107--113, 2012.

\bibitem[SP01]{Schulze}
R.~Schulze-Pillot.
\newblock On explicit versions of tartakovski's theorem.
\newblock {\em Archiv der Mathematik}, 77(2):129--137, 2001.

\end{thebibliography}

\end{document}